\documentclass[11pt]{amsart}

\usepackage{amsfonts, amsmath, amscd}
\usepackage[psamsfonts]{amssymb}

\usepackage{amssymb}

\usepackage{pb-diagram}

\usepackage[usenames]{color}

\input xy
\xyoption{all}

\newtheorem{theorem}{Theorem}[section]
\newtheorem{lemma}[theorem]{Lemma}
\newtheorem{corollary}[theorem]{Corollary}
\newtheorem{question}[theorem]{Question}
\newtheorem{remark}[theorem]{Remark}
\newtheorem{proposition}[theorem]{Proposition}
\newtheorem{definition}[theorem]{Definition}
\newtheorem{example}[theorem]{Example}

\newtheorem{claim}[theorem]{Claim}
\newtheorem{notation}[theorem]{Notation}

\numberwithin{equation}{section}

\newcommand{\w}{\omega}
\newcommand{\NN}{\mathbb{N}}
\newcommand{\DD}{\mathcal{D}}

\newcommand{\dd}{(\mathbf{D})}

\newcommand{\UU}{\mathcal{U}}
\newcommand{\FF}{\mathcal{F}}
\newcommand{\KK}{\mathcal{K}}

\newcommand{\Nn}{\mathcal{N}}

\newcommand{\nn}{\mathfrak{n}}
\newcommand{\Pp}{\mathfrak{P}}
\newcommand{\A}{\mathcal{A}}

\author{S. S. Gabriyelyan, J. K{\c{a}}kol}

\address{Department of Mathematics, Ben-Gurion University of the Negev, Beer-Sheva, P.O. 653, Israel}
\email{saak@math.bgu.ac.il}

\address{Faculty of Mathematics and Informatics, A. Mickiewicz University\\ $61-614$ Pozna{\'n}, Poland}
\email{kakol@amu.edu.pl}
\thanks{The second  named author were supported by Generalitat Valenciana, Conselleria d'Educaci\'{o}, Cultura i Esport, Spain, Grant PROMETEO/2013/058.}

\begin{document}


\title[On $\Pp$-spaces and related concepts]{On $\Pp$-spaces and related concepts}




\begin{abstract}
The concept of the strong Pytkeev property, recently introduced by Tsaban and Zdomskyy in \cite{boaz}, was successfully applied to the study of the space $C_c(X)$ of all continuous real-valued functions  with the compact-open topology on some classes of topological spaces $X$ including \v{C}ech-complete Lindel\"{o}f spaces. Being motivated also by several results providing various concepts of networks we introduce the class of $\Pp$-spaces strictly included in the class of $\aleph$-spaces.  This class of generalized metric spaces is closed under taking subspaces, topological sums and countable products and any space from this class has  countable tightness. Every  $\Pp$-space $X$ has the strong Pytkeev property. The main result of the present paper states  that if $X$ is an $\aleph_0$-space  and $Y$ is a  $\Pp$-space, then the function space $C_c(X,Y)$ has the strong Pytkeev property. This implies that for a separable metrizable space $X$ and  a metrizable topological group $G$  the space $C_c(X,G)$ is metrizable if and only if it is Fr\'{e}chet-Urysohn.
We show that a locally precompact group $G$ is a  $\Pp$-space if and only if $G$ is metrizable.
\end{abstract}

\maketitle

\section{Introduction}

All topological spaces are assumed to be Hausdorff. Various topological properties generalizing metrizability have been studied intensively by topologists and analysts, especially like first countability, Fr\'{e}chet--Urysohness, sequentiality and countable tightness (see \cite{Eng,kak}).  Pytkeev \cite{Pyt} proved that every sequential
space satisfies the property, known actually as the {\it Pytkeev property}, which is stronger than countable
tightness: a topological space $X$ has the {\it Pytkeev property} if for each $A\subseteq X$ and each $x\in \overline{A}\setminus A$, there are infinite subsets $A_1, A_2, \dots $ of $A$ such that each neighborhood of $x$ contains some $A_n$. Tsaban and Zdomskyy \cite{boaz} strengthened this property as follows. A  topological space $X$ has  the {\it strong Pytkeev property} if for each $x\in X$, there exists a countable family $\DD$ of subsets of $Y$,   such that for each neighborhood $U$ of $x$ and  each $A\subseteq X$ with $x\in \overline{A}\setminus A$, there is $D\in\DD$ such that $D\subseteq U$ and $D\cap A$ is infinite. Generalizing the property of the family $\DD$, Banakh in \cite{Banakh} introduced the notion of the Pytkeev network in $X$. A family $\mathcal{N}$ of subsets of a topological space $X$ is called a  {\em Pytkeev network at a point $x\in X$} if $\Nn$ is a network at $x$ and for every open set $U\subset X$ and a set $A$ accumulating at $x$ there is a set $N\in\Nn$ such that $N\subset U$ and $N\cap A$ is infinite; $\Nn$ is  a  {\em Pytkeev network} in $X$ if $\mathcal{N}$ is a Pytkeev network at each point $x\in X$. Hence $X$ has the strong Pytkeev property if and only if $X$ has a countable Pytkeev network at each point $x\in X$.

Now the main result of \cite{boaz} states that the space $C_c(X)$ of all continuous real-valued functions on a Polish space  $X$ (more generally, a separable metrizable space $X$, see \cite{Banakh} or Corollary \ref{c-Banakh} below) endowed with the compact-open topology has the strong Pytkeev property. This result was essentially strengthened in \cite{GKL2}: The space $C_c(X)$ has the strong Pytkeev property for every \v{C}ech-complete Lindel\"{o}f space $X$.
For the proof of this result the authors constructed a family $\DD$ of sets in $C_c(X)$ such that for every neighborhood $U_\mathbf{0}$ of the zero function $\mathbf{0}$ the union $\bigcup \{ D \in\DD : \mathbf{0}\in D \subseteq U_\mathbf{0}\}$ is a  neighborhood of $\mathbf{0}$ (see the condition $\dd$ in \cite{GKL2}).
Being inspired by this idea for $C_{c}(X)$  we propose the following  types of networks which will be applied in the sequel.

\begin{definition} \label{def2} {\em
A family $\Nn$ of subsets of a topological space $X$ is called

$\bullet$ a {\em $cn$-network}  at a point $x\in X$ if for each neighborhood $O_x$ of $x$ the set $\bigcup \{ N \in\Nn : x\in N \subseteq O_x \}$ is a neighborhood of $x$; $\Nn$ is a {\em $cn$-network} in $X$ if $\mathcal{N}$ is a $cn$-network at each point $x\in X$.

$\bullet$  a {\em $ck$-network}  at  a point $x\in X$ if for any neighborhood $O_x$ of $x$ there is a neighborhood $U_x$ of $x$ such that for each compact subset $K\subset U_x$ there exists a finite subfamily $\FF\subset\mathcal{N}$ satisfying $x\in \bigcap\FF$ and $K\subset\bigcup\FF\subset O_x$; $\Nn$ is a {\em $ck$-network}  in $X$ if $\mathcal{N}$ is a $ck$-network at each point $x\in X$.

$\bullet$  a {\em $cp$-network}  at  a point  $x\in X$ if either $x$ is an isolated point of $X$ and $\{x \} \in\Nn$, or for each subset $A\subset X$ with $x\in \overline{A}\setminus A$ and each neighborhood $O_x$ of $x$ there is a set $N\in\mathcal{N}$ such that $x\in N\subset O_x$ and $N\cap A$ is infinite; $\Nn$ is  a  {\em $cp$-network} in $X$ if $\mathcal{N}$ is a Pytkeev network at each point $x\in X$.  }
\end{definition}

These notions relate as follows:
\[
\xymatrix{\mbox{base (at $x$)}\ar@{=>}[r]&
\mbox{ $ck$-network (at $x$)}\ar@{=>}[r]&\mbox{$cn$-network (at $x$)}\ar@{=>}[r]&\mbox{network (at $x$)}.
}
\]
The following fact (see Proposition \ref{pTigh}) additionally explains our interest to the study of spaces $X$ with countable $cn$-network at each point $x\in X$: If $X$ has a countable $cn$-network at a point $x$ then $X$ has a countable tightness at $x$. In Section \ref{secRel} we recall other important types of networks and related results used in the article.

Let us recall the following classes of topological spaces admitting certain countable networks of various types.
\begin{definition} {\em
A topological space $X$ is said to be
\begin{itemize}
\item (\cite{Mich}) a {\em cosmic} space if $X$ is regular and has a countable network;
\item (\cite{Mich})  an {\em $\aleph_0$-space} if $X$ is regular and  has a countable $k$-network;
\item (\cite{Banakh})  a {\em $\Pp_0$-space} if $X$ is regular and  has a countable Pytkeev network.
\end{itemize} }
\end{definition}
It is known also that:  $\Pp_0$-space $\Rightarrow$ $\aleph_0$-space $\Rightarrow$ cosmic, but the converse is  false (\cite{Banakh,Mich}).

It is easy to see that for each network (resp. each $k$-network or a Pytkeev network) $\Nn$ in a topological space $X$ the family $\mathcal{N}\vee\mathcal{N} :=\{ A\cup B: A,B\in\mathcal{N}\}$ is a $cn$-network (resp. a $ck$-network or a $cp$-network) in $X$. Hence, a regular space $X$ is cosmic (resp. an $\aleph_0$-space or a  $\Pp_0$-space) if and only if $X$ has a countable $cn$-network (resp.  a countable $ck$-network or a countable $cp$-network).
On the other hand, the versions of network and $cn$-network differ at the $\sigma$-locally finite level.

Okuyama \cite{Oku} and O'Meara \cite{OMe2}, having in mind the Nagata-Smirnov metrization
theorem, introduced the classes of $\sigma$-spaces and $\aleph$-spaces, respectively, which contain all
metrizable spaces.
\begin{definition} \label{def3} {\em
A topological space $X$ is called

$\bullet$ (\cite{Oku}) a {\em $\sigma$-space} if $X$ is regular and has a $\sigma$-locally finite network.

$\bullet$ (\cite{OMe2}) an {\em $\aleph$-space} if $X$ is regular and has a $\sigma$-locally finite $k$-network.}
\end{definition}
This motivates us to propose the following concept.
\begin{definition}  \label{def4}{\em
A topological space $X$ is called   a  {\em $\Pp$-space} if $X$ has a $\sigma$-locally finite $cp$-network.}
\end{definition}
Each  $\Pp$-space $X$ has the strong Pytkeev property (see Corollary \ref{c-Aleph}).

As one should  expect, any  $\Pp$-space is an $\aleph$-space. Moreover, it turns out that  $\Pp$-spaces satisfy even a stronger condition. In Section \ref{secGMS} we study also the following {\it strict} versions of $\sigma$-spaces and $\aleph$-spaces.
\begin{definition}  \label{def5}{\em
A topological space $X$ is called

$\bullet$ a {\em strict $\sigma$-space} if $X$ has a $\sigma$-locally finite $cn$-network;

$\bullet$ a {\em strict $\aleph$-space} if $X$ has a $\sigma$-locally finite $ck$-network.}
\end{definition}

The following diagram describes the relation between new, as well as, known classes of generalized metric spaces and justifies the study of strict $\sigma$-spaces and strict $\aleph$-spaces.
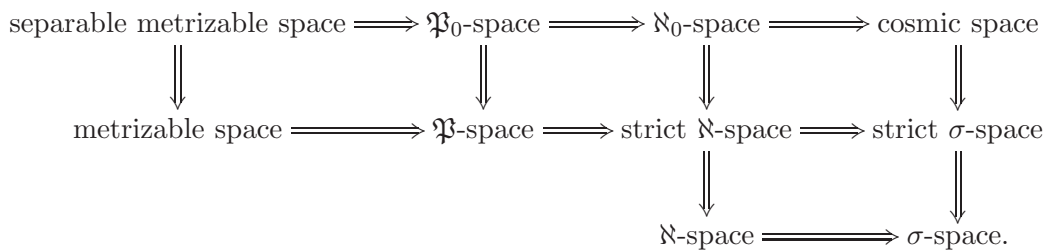
\[
\xymatrix{
\mbox{separable metrizable space}\ar@{=>}[r]\ar@{=>}[d]&
\mbox{$\Pp_0$-space}\ar@{=>}[r]\ar@{=>}[d] &\mbox{$\aleph_0$-space}\ar@{=>}[r]\ar@{=>}[d]&\mbox{cosmic space}\ar@{=>}[d]\\
\mbox{metrizable space}\ar@{=>}[r]&\mbox{$\Pp$-space}\ar@{=>}[r]&\mbox{strict $\aleph$-space}\ar@{=>}[r]\ar@{=>}[d] &\mbox{strict $\sigma$-space}\ar@{=>}[d]\\
&&\mbox{$\aleph$-space}\ar@{=>}[r]&\mbox{$\sigma$-space.}
}
\]
None of the implications in this diagram can be reversed (see Theorem \ref{tMetPrecomGr} and Examples \ref{exa-A-non-CS} and \ref{exaBohr}).

In Section \ref{secMetr} we propose some  criterions for metrizability of topological spaces (see Theorem \ref{tMetr-Pyt}).
For the precompact topological groups, i.e. which are embedded into a locally compact group,  we prove the following.
\begin{theorem} \label{tMetPrecomGr}
For a locally precompact topological group $G$ the following conditions are equivalent:
\begin{enumerate}
\item[{\rm (i)}]  $G$ is metrizable.
\item[{\rm (ii)}]  $G$ is a  $\Pp$-space.
\item[{\rm (iii)}]  $G$ has the strong Pytkeev property.
\end{enumerate}
\end{theorem}

It is well-known that the classes of $\aleph$-spaces and $\sigma$-spaces are closed under taking subspaces, topological sums and countable products (see \cite{gruenhage}). The same holds also for the new three classes of generalized metric spaces introduced in Definitions \ref{def4} and  \ref{def5}: they are closed under taking subspaces, topological sums and countable products (see Section \ref{secPro}). It is well-known that the class of $\aleph$-spaces is closed also under taking  function spaces with the Lindel\"{o}f domain. Given topological spaces $X$ and $Y$, let $C(X,Y)$ be the family of all continuous functions from $X$ into $Y$, and denote by $C_c(X,Y)$  the family $C(X,Y)$ endowed with the compact-open topology. Foged \cite{foged} (and O'Meara \cite{OMe2}) proved that $C_c(X,Y)$ is an $\aleph$-space for each $\aleph_0$-space $X$ and any $\aleph$-space $Y$. So it is natural to ask whether an analogous result holds also for  $\Pp$-spaces $Y$. In the last Section \ref{secFunc} we prove the following partial result which is the main result of the article.
\begin{theorem} \label{t-strong-Pyt-A}
Let $X$ be an $\aleph_0$-space. Then:
\begin{enumerate}
\item If $Y$ is a $\Pp$-space, then the function space $C_c(X,Y)$ has the strong Pytkeev property.
\item If $Y$ is a strict $\aleph$-space, then the function space $C_c(X,Y)$ has a countable $ck$-network at each function $f\in C_c(X,Y)$.
\end{enumerate}
\end{theorem}
This implies that for a separable metrizable space $X$
and a metrizable topological group $G$  the space $C_c(X,G)$ is metrizable if and only if it is
Fr\'echet-Urysohn, see Corollary \ref{nice}.
Note that item (2) of this theorem shows additionally a big difference between strict $\aleph$-spaces and  $\aleph$-spaces: any $\aleph$-space $X$ (in particular, a function space $C_c(X,Y)$ with an $\aleph_0$-space $X$ and an $\aleph$-space $Y$) has only countable $cs^\ast$-character (see \cite{GK2} or Corollary \ref{cAleph-cs}), but $X$ may not have a countable $k$-network at some point $x\in X$ by Example \ref{exa-A-non-CS}. Also in Section \ref{secFunc}  we pose  several open problems.

Notice that the next our paper \cite{GK-GMS2}  describes the topology of a topological space $X$ admitting  a countable $cp$-, $ck$- or $cn$-network at a point $x\in X$, and  also provides some applications for topological groups and topological vector spaces.

\section{Networks in topological spaces and relations between them} \label{secRel}

Recall the most important types of networks which are used in the paper.
\begin{definition} \label{def1} {\em
Let $\mathcal{N}$ be a family  of subsets of a topological space $X$. Then:

$\bullet$ (\cite{Arhan}) $\Nn$ is a {\em network at a point} $x\in X$ if for each neighborhood $O_x$ of $x$ there is a set $N\in\mathcal{N}$ such that $x\in N\subset O_x$; $\Nn$ is a {\em network} in $X$ if $\mathcal{N}$ is a network at each point $x\in X$.

$\bullet$  (\cite{Guth}) $\Nn$ is a {\em $cs$-network at  a point} $x\in X$ if for each sequence $(x_n)_{n\in\NN}$ in $X$ convergent to  $x$ and for each neighborhood $O_x$ of $x$ there are $N\in\mathcal{N}$ and $k\in\NN$ such that $\{ x\} \cup \{x_n : n\geq k\} \subset N\subset O_x$; $\Nn$ is a {\em $cs$-network } in $X$ if $\mathcal{N}$ is a $cs$-network at each point $x\in X$.

$\bullet$ (\cite{Gao}) $\Nn$ is a {\em $cs^\ast$-network at  a point} $x\in X$ if for each sequence $(x_n)_{n\in\NN}$ in $X$ converging to  $x$ and for each neighborhood $O_x$ of $x$ there is a set $N\in\mathcal{N}$ such that $x\in N\subset O_x$ and the set $\{n\in\NN :x_n\in N\}$ is infinite; $\Nn$ is a {\em $cs^\ast$-network}  in $X$ if $\mathcal{N}$ is a $cs^\ast$-network at each point $x\in X$.

$\bullet$ (\cite{Banakh1}) $\Nn$ is  a {\em local} {\em $k$-network at a point $x\in X$} if for each neighborhood $O_x$ of $x$ there is a neighborhood $U_x$ of $x$ such that for each compact subset $K\subset U_x$ there is a finite subfamily $\FF\subset\mathcal{N}$ such that $K\subset\bigcup\FF\subset O_x$; $\Nn$ is a {\em local $k$-network} in $X$ if $\mathcal{N}$ is a local $k$-network at each point $x\in X$.

$\bullet$ (\cite{Mich}) $\Nn$ is a {\em $k$-network} in $X$ if whenever $K\subset U$ with $K$ compact and $U$ open in $X$, then $K\subset \bigcup\FF\subset U$ for some finite $\FF\subset\Nn$.
}
\end{definition}
For regular spaces $X$ the notions of local $k$-network and $k$-network coincide (see Remark \ref{r1} below). Note that a regular space $X$ is an $\aleph_0$-space if and only if $X$ has a countable $cs$-network (\cite{Guth}) if and only if $X$ has a countable $cs^\ast$-network  (\cite{Gao}).

Below we discuss some simple relations between various types of networks.
\begin{remark} \label{r1} {\em
Let $X$ be  a topological space.


(i) Each $ck$-network (at a point $x\in X$) is a local $k$-network (at $x$). On the other hand, if $\Nn$ is a local $k$-network at $x$, then the family $\Nn_x :=\{ N\cup\{ x\} : N\in\Nn\}$ is a $ck$-network at $x$.
Also if $\Nn$ is a $k$-network in $X$, then $\Nn$ is a local $k$-network and the family $\mathcal{N}\vee\mathcal{N}$ is a $ck$-network for $X$.

(ii) If $X$ is a {\em regular} space and $\Nn$ is a local $k$-network for $X$, then $\Nn$ is a $k$-network. Indeed, let $K\subset U$ with $K$ compact and $U$ open in $X$. For every $x\in K$, take  open neighborhoods $W_x, V_x$ and $O_x$ of $x$ such that $\overline{W_x}\subset V_x\subset O_x\subset U$ and $V_x$ satisfies the definition of $k$-network for $O_x$. So there is a finite family $\FF_x \subset \Nn$ such that  $\overline{W_x} \cap K \subset \bigcup\FF_x \subset U_x$. Since $K$ is compact, $K\subset \bigcup_{j=1}^m W_{x_j}$ for some $x_1,\dots,x_m \in K$. Clearly,
\[
K\subset \bigcup_{j=1}^m \bigcup\FF_{x_j} \subset  \bigcup_{j=1}^m O_{x_j} \subset U.
\]
Thus $\Nn$ is a $k$-network.

(iii) It is clear that any $cp$-network $\Nn$ at a point $x\in X$ is  a $cs^\ast$-network at $x$.
Observe that $\Nn$ is also a $cn$-network at $x$. Indeed, let $U$ be a neighborhood of $x$ and $W:= \bigcup \{ N\in\Nn : N\subset U\}$. We have to prove that $W$ is a neighborhood of $x$. If this is not the case, then $x\in \overline{X\setminus W}$. By definition, there is $N\in\Nn$ such that $x\in N\subset U$ and $N\cap(X\setminus W)$ is infinite. Since $N\subset W$ we obtain $W\cap (X\setminus W)\not= \emptyset$, a contradiction.

(iv) If $\mathcal{N}$ is a $ck$-network (at a point $x\in X$) in $X$, then $\mathcal{N}$ is a $cn$-network and a $cs^\ast$-network (at $x$).

(v) For any topological space $X$ the family $\{ \{x\}: x\in X\}$ is trivially a network for $X$. So to avoid such a trivial and unpleasant situation in which a network actually has nothing common with the {\it topology} of $X$, we guess that the notion of $cn$-network is of interest.

(vi) If $\Nn$ is a Pytkeev network at a point $x$ of  $X$, then  the family $\Nn_x =\{ N\cup\{x\}: N\in\Nn\}$ is a $cp$-network at $x$. So the difference between these notions is not essential when they are considered only at a {\it fix} point (as in the definition of the strong Pytkeev property). But these notions essentially differ on $\sigma$-locally finite level, see Example \ref{exa-A-non-CS} below.
}
\end{remark}

Recall that a topological space $X$ has {\it countable tightness at a point $x\in X$} if whenever $x\in \overline{A}$ and $A\subseteq X$, then $x\in \overline{B}$ for some countable $B\subseteq A$; $X$ has {\em countable tightness} if it has countable tightness at each point $x\in X$.

\begin{proposition} \label{pTigh}
Let a topological space $X$ have a countable $cn$-network at a point $x$. Then $X$ has countable tightness at $x$.
\end{proposition}
\begin{proof}
Let $\{ D_n\}_{n\in\NN}$ be a countable $cn$-network at $x$ and $A\subset X$ be such that $x\in \overline{A}\setminus A$. Set $J:=\{ n\in\NN : D_n \cap A \not=\emptyset\}$. For every $n\in J$ take arbitrarily $a_n\in D_n \cap A$ and set $B:= \{ a_n \}_{n\in J} \subset A$. We show that $x\in \overline{B}$.   For every neighborhood $U$ of $x$ set $I(U):=\{ n\in\NN : x\in D_n \subseteq U\}$. Then, by definition, the set $\bigcup_{n\in I(U)} D_n$ contains a neighborhood $V$ of $x$. Since $A\cap V\not=\emptyset$, we can find $n\in I(U)\cap J$. Thus $a_n \in B\cap U$.
\end{proof}

It is natural to ask for which topological spaces some of types of networks coincide? Partial answers to this question are given in three  propositions below proved by Banakh \cite{Banakh1}. 
\begin{proposition}[\cite{Banakh1}]\label{fB}
Any countable Pytkeev network  at a point $x$ of a topological space $X$ is a local k-network  at $x$. Consequently, any countable $cp$-network at a point $x\in X$ is a $ck$-network at $x$.
\end{proposition}

Recall that a topological space $X$ is called a {\it $k$-space} if for each non-closed subset $A\subset X$ there is a compact subset $K\subset X$ such that $K\cap A$ is not closed in $K$.
\begin{proposition}[\cite{Banakh1}] \label{fB2}
If a topological space $X$ is a regular $k$-space and $\mathcal{N}$ is a $k$-network at a point $x\in X$, then  $\mathcal{N}\vee\mathcal{N} :=\{ C\cup D: C,D\in\mathcal{N}\}$ is a Pytkeev network at $x$.
\end{proposition}



\begin{corollary} \label{c-A0=PA0}
A $k$-space $X$ is a $\Pp_0$-space if and only if $X$ is an $\aleph_0$-space.
\end{corollary}

Since any topological group is regular, Propositions \ref{fB} and \ref{fB2} yield
\begin{corollary}\label{c-k-CP=CK}
Let $G$ be a topological group which is a $k$-space. Then $G$ has countable $cp$-network at the unit $e$ if and only if $G$ has a countable $ck$-network at $e$.
\end{corollary}
Note that the condition to be a $k$-space in Corollary \ref{c-k-CP=CK} is essential, see Theorem \ref{tMetPrecomGr}
and Example \ref{exaBohr}.

Recall that a topological space $X$ is called {\it Fr\'{e}chet-Urysohn at a point $x\in X$} if for every subset $A\subset X$ such that $x\in \overline{A}$ there exists a sequence $\{ a_n\}_{n\in\NN}$ in $A$ converging to $x$; $X$ is called {\it Fr\'{e}chet-Urysohn} if it is Fr\'{e}chet-Urysohn at each point $x\in X$.
\begin{proposition}[\cite{Banakh1}]  \label{fB3}
Let $\Nn$ be a $cs^\ast$-network at a point $x$ of a topological space $X$. If  $X$ is Fr\'{e}chet-Urysohn at $x$, then $\Nn$ is a Pytkeev network (actually a $cp$-network) at $x$.
\end{proposition}


\begin{corollary}
In the class of Fr\'{e}chet-Urysohn spaces the concepts of $cp$-network and $cs^\ast$-network are equivalent.
\end{corollary}

To unify notations let us call $\sigma$-spaces by       $0\mbox{-}\sigma$-spaces, and networks by $0$-networks.
\begin{notation} {\em
If $\Nn$ is either a $ck$-,  $cs$-, $cs^\ast$-, $cp$-,  $cn$-network or $0$-network (at a point $x$) in a topological space $X$, we will say  that $\Nn$ is an {\it $\mathfrak{n}$-network} (at $x$).
Set $\mathfrak{N}=\{ ck,  cs, cs^\ast, cp, cn, 0\}$.}
\end{notation}

Definitions \ref{def1} and \ref{def2} allow us to define the following cardinals of topological spaces.
\begin{definition} {\em
Let $x$ be  a point of a topological space $X$ and $\nn\in\mathfrak{N}$. The smallest size $|\mathcal{N}|$ of an $\nn$-network $\mathcal{N}$ at $x$ is called the {\em $\nn$-character of $X$ at the point $x$} and is denoted by $\nn_\chi(X,x)$. The cardinal $\nn_\chi(X)=\sup\{ \nn_\chi(X,x): x\in X\}$ is called the {\em $\nn$-character} of  $X$. The {\em $\nn$-netweight}, $\nn w(X)$, of $X$ is the least cardinality of an $\nn$-network for $X$.}
\end{definition}
Analogously we  define the local $k$-character (at a point $x$) of a topological space $X$.

In the paper we study topological spaces $X$ with countable $\nn$-character (at a point $x\in X$), i.e. spaces $X$ with $\nn_\chi(X)\leq\aleph_0$ (respectively, $\nn_\chi(X,x)\leq\aleph_0$). Recall again  (see \cite{boaz})  that a topological space $X$ has the {\em strong Pytkeev property} if and only if $X$ has a countable Pytkeev network at each point $x\in X$, i.e. if $cp_\chi(X)\leq \aleph_0$.
If a space $X$ is first countable at a point $x\in X$, then any countable base at $x$ is a $cp$-network at $x$. So, every  first countable space $X$ has the strong Pytkeev property.

As usual we denote by $\chi(X,x)$ the character of a topological space $X$ at a point $x$, and the character of $X$ is denoted by $\chi(X)$.
Applying the definition of the  $cn$-network we have the following
\begin{proposition} \label{pCharacter}
If $x$ is a point of a topological space $X$, then $\chi(X,x)\leq 2^{cn_\chi(X,x)}$.
\end{proposition}

\begin{example}{\em
Let $G$ be a discrete abelian group of cardinality $2^\kappa$, where the cardinal $\kappa$ is uncountable. It is well known that $\chi(G, \tau_b)=2^{|G|}$, where $\tau_b$ is the Bohr topology of $G$. By Proposition \ref{pCharacter} we have $\aleph_0 <\kappa \leq  cn_\chi(G, \tau_b)$. On the other hand, the group $(G, \tau_b)$ has countable tightness by  \cite[Problem 9.9.H]{ArT}. So there are precompact abelian topological groups of countable tightness with arbitrary large $cn$-character. Since any convergent sequence in $(G, \tau_b)$ is essentially constant, we have $cs^{\ast}_\chi (G, \tau_b)=1 < cn_\chi (G, \tau_b).$ }
\end{example}

Proposition \ref{fB} and Remark \ref{r1} imply
\begin{corollary} \label{c1}
Let $x$ be a point of a topological space $X$. Then
\[
cp_\chi(X,x)\leq\aleph_0 \Longrightarrow ck_\chi(X,x)\leq\aleph_0 \Longrightarrow \max\{ cs^{\ast}_\chi(X,x), cn_\chi(X,x) \} \leq\aleph_0;
\]
\[
cn_\chi(X,x)\leq \min\{ cp_\chi(X,x), ck_\chi(X,x)\}, \quad  cn_\chi(X)\leq \min\{ cp_\chi(X), ck_\chi(X)\};
\]
\[
\begin{split}
cs^\ast_\chi(X,x) & \leq \min\{ cp_\chi(X,x), ck_\chi(X,x), cs_\chi(X,x)\} \\
& \leq \max\{ cp_\chi(X,x), ck_\chi(X,x), cs_\chi(X,x)\} \leq \chi(X,x);
\end{split}
\]
\[
cs^\ast_\chi(X)\leq \min\{ cp_\chi(X), ck_\chi(X), cs_\chi(X)\} \leq \max\{ cp_\chi(X), ck_\chi(X), cs_\chi(X)\} \leq \chi(X).
\]
\end{corollary}


\section{Three new types of generalized metric spaces} \label{secGMS}


Recall (\cite{Guth2}) that a topological space $X$ is called a {\em $cs\mbox{-}\sigma$-space} if it is regular and has a $\sigma$-locally finite $cs$-network. Analogously  we define
\begin{definition} \label{def-cn-space} {\em
For $\nn\in\mathfrak{N}$, a topological space $X$ is called an {\em $\nn\mbox{-}\sigma$-space} if it is regular and has a $\sigma$-locally finite $\nn$-network.}
\end{definition}%
So $cp\mbox{-}\sigma$-spaces are $\Pp$-spaces, $ck\mbox{-}\sigma$-spaces are strict $\aleph$-spaces, $cn\mbox{-}\sigma$-spaces are strict $\sigma$-spaces and $0\mbox{-}\sigma$-spaces are $\sigma$-spaces, respectively.

\begin{remark} {\em
Each $ck\mbox{-}\sigma$-space is both a $cn\mbox{-}\sigma$-space and a $cs^\ast\mbox{-}\sigma$-space.}
\end{remark}

Nagata-Smirnov metrization theorem implies
\begin{proposition} \label{pMetr-Pyt}
Any metrizable space $X$ is a  $\Pp$-space. Each separable metrizable space is a  $\Pp_0$-space.
\end{proposition}
\begin{proof}
By \cite[4.4.4]{Eng},  $X$ has a $\sigma$-locally finite open base $\DD$. Clearly, $\DD$ is also a $cp$-network for $X$. If additionally $X$ is separable,  it is clear that any countable open base of $X$ is  a countable $cp$-network for $X$.
\end{proof}

It is known  (see  \cite{Gao}) that each countable $cs^\ast$-network in a regular space $X$ is a $k$-network.  Next proposition  generalizes this fact.
\begin{proposition} \label{pA-cs-ck}
Any $\sigma$-locally finite $cs^\ast$-network in a regular topological space $X$ is a $k$-network for $X$. Consequently, each countable $cs^\ast$-network in $X$ is a $k$-network.
\end{proposition}
\begin{proof}
Let $\DD =\bigcup_n \DD_n$ be an increasing $\sigma$-locally finite $cs^\ast$-network for $X$  and let $K$ be a compact subset of an open set $U\subset X$. We have to find an $n\in\NN$ and a finite subfamily $\FF$ of $\DD_n$ such that $K\subset\bigcup\FF \subset U$.

For each $x\in K$ and every $n\in\NN$ take an open neighborhood $U_n(x)$ of $x$ such that $U_n(x)\subset U$ and $U_n(x)$ intersects with only finite subfamily $T_n(x)$ of $\DD_n$. Set
\[
R_n(x):= \{ D\in T_n(x): \ D\subset U \mbox{ and } D\cap K\not=\emptyset \}.
\]
Since $K$ is compact there are $x^n_1,\dots,x^n_{s_n} \in K$ such that $K\subset \cup_{i=1}^{s_n} U_n(x^n_i)$. Set
\[
A_n := \bigcup_{j=1}^n \bigcup_{i=1}^{s_j} \left\{ D: \, D\in R_j(x^j_i) \right\}.
\]
Clearly,  $A_1 \subset A_2 \subset\dots \subset U$. Since $\DD$ is a network for $X$, it is enough to show that $K\setminus A_n$ is finite for some $n\in\NN$.

Suppose for a contradiction that $K\setminus A_n$ is infinite for every $n\in\NN$. Then we can choose a sequence $\{ x_n\}$ of distinct elements of $K$ such that $x_n\not\in A_n$ for every $n\in\NN$. By \cite[Corollary 4.7]{gruenhage}, $K$ is metrizable. So without loss of generality we may assume that  $\{ x_n\}$ converges to a point $z\in K$. As $\DD$ is a $cs^\ast$-network, there is $q\in \NN$ and $D\in \DD_q$ such that $z\in D\subset U$ and $D\cap \{ x_n\}$ is infinite. If an index $i$ is such that $z\in U_q(x^q_i)$, then $D\in R_q(x^q_i)$, so $D\subset A_q$. By the construction of $ \{ x_n\}$, we have $x_n\not\in D$ for every $n\geq q$. Thus $D\cap \{ x_n\}$ is finite, a contradiction.
\end{proof}

Since any $cp$-network is trivially a $cs^\ast$-network, Proposition \ref{pA-cs-ck} implies that any $\Pp_0$-space is an $\aleph_0$-space (see \cite{Banakh}).

Next theorem shows that some types of $\sigma$-locally finite networks coincide.
\begin{theorem} \label{tA-cs}
For a topological space $X$ the following assertions are equivalent
\begin{enumerate}
\item[{\rm (i)}] $X$ is an $\aleph$-space;
\item[{\rm (ii)}] {\em (\cite{foged})} $X$ is a $cs\mbox{-}\sigma$-space;
\item[{\rm (iii)}] {\em (see \cite{Nagata})} $X$ is a $cs^\ast\mbox{-}\sigma$-space.
\end{enumerate}
\end{theorem}

\begin{proof}
(i)$\Leftrightarrow$(ii) was proved by Foged \cite{foged}. (ii)$\Rightarrow$(iii) is trivial, and (iii)$\Rightarrow$(i) follows from Proposition \ref{pA-cs-ck} (this also follows from Lemma 1.17 and Theorem 1.4 of \cite{Nagata}).
\end{proof}

The following theorem is essentially used in the proof of Theorem \ref{t-strong-Pyt-A}.
\begin{theorem} \label{t-Lind-Set}
For $\nn\in\{ 0, cn, cs^\ast, ck, cp\}$, let $X$ be an $\nn\mbox{-}\sigma$-space and $\DD=\bigcup_{n\in\NN} \DD_n$ be a closed increasing $\sigma$-locally finite $\nn$-network for  $X$. Then for every Lindel\"{o}f subset $S$ of $X$ there exists a sequence $\Nn=\{ D_k\}_{k\in\NN}\subset \DD$  which is an $\nn$-network at each point of $S$.

In the case $\nn\in\{ ck,cp\}$, the family $\Nn$ satisfies additionally the following property: if $K\subset S\cap U$ with $K$ compact and $U$ open, then there is an open subset $W$ of $X$ such that
\begin{itemize}
\item[{\rm (a)}] $K\subset W \subset \bigcup_{k\in I(U)} D_k \subset U$, where $I(U)=\{ k\in\NN : D_k\subset U\}$, and
\item[{\rm (b)}]  for each compact subset $C$ of $W$ there is a finite subfamily $\alpha$ of $I(U)$ for which $C\subset \bigcup_{k\in\alpha} D_k$.
\end{itemize}
\end{theorem}

\begin{proof}
Since $\DD$ contains closed sets, for every $x\in S$ and every $n\in\NN$ choose a neighborhood $U_n(x)$ of $x$ such that $U_n(x)$ intersects with a finite subfamily $T_n(x)$ of $\DD_n$ such that
\begin{equation} \label{eB}
T_n(x)= \{ D\in\DD_n : D\cap U_n(x) \not=\emptyset \} = \{ D\in\DD_n : x\in D\}.
\end{equation}
So, for every $x\in S$ the set
\[
T(x):= \{ D\in\DD : \ x\in D \} = \bigcup_{n\in\NN}  T_n(x)
\]
is a countable $\nn$-network at $x$.

As $S$ is Lindel\"{o}f, for every $n\in\NN$ there is a sequence $\{ x_j^n\}_{j\in\w} \in S$ such that the set $\bigcup_{j\in\omega} U_n(x_j^n)$ is an open neighborhood of $S$. Let $\Nn=\{ D_k\}_{k\in\NN}$ be an enumeration of the family $\bigcup_{n\in\NN} \bigcup_{j\in\w} T(x_j^n)\subset \DD$. We show that $\Nn$ is as desired.

(1) Assume that $\DD$ is a network.  Fix arbitrarily $x\in S$ and an open neighborhood $O_x$ of $x$. Then there is $n\in\NN$ and $D\in\DD_n$ such that $x\in D\subset O_x$. Choose $j_0\in\w$ such that $x\in U_n(x_{j_0}^n)$. Then $D\cap U_n(x_{j_0}^n) \not=\emptyset$; so, by (\ref{eB}), $D\in T_n(x_{j_0}^n)$. Thus $D\in \Nn$.

Assume in addition that $\DD$ is a $cn$-network. We showed above that for every $D\in\DD$ with $x\in D\subset O_x$ there are $n\in\NN$ and $j_0\in\w$ such that $D\in T_n(x_{j_0}^n)\subset \Nn$. Then
\[
\bigcup\{ D_k\in\Nn : x\in D_k \subset O_x\} = \bigcup\{ D\in\DD : x\in D \subset O_x\}
\]
is a neighborhood of $x$ by the definition of $cn$-network.

(2) Assume that $\DD$ is a $cs^\ast$-network.  Fix arbitrarily $x\in S$, an open neighborhood $O_x$ of $x$ and a sequence $\{ x_n\}_{n\in\w}$ converging to $x$. Since $\DD$ is a $cs^\ast$-network in $X$, there is $n\in\NN$ and $D\in\DD_n$ such that $x\in D\subset O_x$ and the set $\{n\in\w : x_n \in D\}$ is infinite.  Choose $j_0\in\w$ such that $x\in U_n(x_{j_0}^n)$. Then $D\cap U_n(x_{j_0}^n) \not=\emptyset$; so, by (\ref{eB}), $D\in T_n(x_{j_0}^n)$. Thus $D\in \Nn$.

(3) Assume that $\DD$ is a $ck$-network. Fix arbitrarily $x\in S$ and an open neighborhood $O_x$ of $x$. Since $\DD$ is a $ck$-network at $x$, there exists a neighborhood $V$ of $x$ such that $V\subset O_x$ and for each compact subset $C$ of $V$ there are $n\in\NN$ and  a finite subset $\FF$ of $\DD_n$ such that $x\in \bigcap\FF$ and $C\subset\bigcup\FF\subset V$. Choose $j_0\in\w$ such that $x\in U_n(x_{j_0}^n)$. Then $D\cap U_n(x_{j_0}^n) \not=\emptyset$ for every $D\in\FF$; so, by (\ref{eB}), $D\in T_n(x_{j_0}^n)$. Thus $D\in \Nn$ and $\FF\subset\Nn$. Therefore $\Nn$ is a $ck$-network at $x$.

(4) Assume that $\DD$ is a $cp$-network. We show that $\Nn$ is a $cp$-network at each point $x\in S$. Fix arbitrary $x\in S$. If $x$ is isolated in $X$, then $\{ x\}\in\DD$ by definition. If $n\in\NN$ and $j\in\w$ are such that $x\in U_n(x^n_j)$, then $\{ x\}\cap U_n(x^n_j)\not=\emptyset$. Then (\ref{eB}) implies that $x=x^n_j$, and hence $\{ x\} \in\Nn$. Assume that $x$ is non-isolated. Fix an open neighborhood $O_x$ of $x$ and a subset $A\subset X$ with $x\in \overline{A}\setminus A$. Since $\DD$ is a $cp$-network in $X$, there is $n\in\NN$ and $D\in\DD_n$ such that $x\in D\subset O_x$ and $D\cap A$ is infinite.  Choose $j_0\in\w$ such that $x\in U_n(x_{j_0}^n)$. Then $D\cap U_n(x_{j_0}^n) \not=\emptyset$; so, by (\ref{eB}), $D\in T_n(x_{j_0}^n)$. Thus $D\in \Nn$.

Now assume  $\nn\in\{ ck,cp\}$. Since any countable $cp$-network at a point $x$ is also a $ck$-network at $x$ by Proposition \ref{fB}, we can assume that $\nn=ck$.  Let $K\subset S\cap U$ with $K$ compact and $U\subset X$ open. For every $x\in K$ the set
\[
I(x):= \{ k\in\NN : \ x\in D_k \subset U\}
\]
is countable. Since $\Nn$ is a $ck$-network at $x$  and hence a $cn$-network at $x$, the set $O(x):=\bigcup \{ D_k: k\in I(x)\}$ is a neighborhood of $x$ and $O(x)\subset U$. By the definition of $ck$-network at $x$ and since $X$ is regular, there are open neighborhoods $W(x)$ and $V(x)$ of $x$ such that $\overline{W(x)}\subset V(x)\subset O(x)$ and for each compact subset $C$ of $V(x)$ there is a finite subset $\FF$ of $\{ D_k : k\in I(x)\}$ with $C\subset\bigcup\FF\subset O(x)$. Since $K$ is compact there are $z_1,\dots, z_s \in K$ such that the set
\[
W:= \bigcup_{j=1}^s W(z_j)
\]
is an open neighborhood of $K$. Clearly, (a) holds by construction. Let us check (b). Let $C$ be an arbitrary compact subset of $W$. Then for every $1\leq j\leq s$, there is a finite subfamily $\alpha_j \subset I(z_j)$ such that
\[
\overline{W(z_j)} \cap C \subset \bigcup_{k\in\alpha_j} D_k \subset O(z_j).
\]
Set $\alpha:=\cup_{j=1}^s \alpha_j$. Then
\[
C =\bigcup_{j=1}^s (\overline{W(z_j)} \cap C) \subset \bigcup_{j=1}^s \bigcup_{k\in\alpha_j} D_k =\bigcup_{k\in\alpha} D_k  \subset   \bigcup_{j=1}^s O(z_j) \subset U,
\]
and hence (b) holds true.
\end{proof}

Next obvious corollary of Theorem \ref{t-Lind-Set} shows that the property to be an $\nn\mbox{-}\sigma$-space is stronger than  the property to have countable $\nn$-character.
\begin{corollary} \label{c-Aleph}
For $\nn\in\mathfrak{N}$, if $X$ is an $\nn\mbox{-}\sigma$-space, then $\nn_\chi(X)\leq\aleph_0$.
\end{corollary}

Theorem \ref{tA-cs} and Corollary \ref{c-Aleph} immediately imply the  following result which might be also noticed from the proof of  \cite[Corollary 2.18]{Sak} and is proved in \cite{GK2}.
\begin{corollary} \label{cAleph-cs}
Each $\aleph$-space $X$ has countable  $cs^\ast$-character.
\end{corollary}

The following example connects two examples kindly proposed us by Taras Banakh. This example shows that the classes of $cn\mbox{-}\sigma$-spaces and $ck\mbox{-}\sigma$-space are much smaller than the classes of $\sigma$-spaces and $\aleph$-spaces respectively.

\begin{example}[Banakh] \label{exa-A-non-CS} {\em
For any uncountable cardinal $\kappa$ there is a paracompact $\aleph$-space $\Omega$ of tightness $\kappa$ with a unique non-isolated point $\infty$ at which $cn_\chi(\Omega,\infty)=\kappa$. In particular, the space $\Omega$ is not a strict $\sigma$-space.}
\end{example}

\begin{proof}
Consider the space $\Omega=\{\infty\}\cup(\kappa\times \w\times \w)$ in which all points $x\in\kappa\times\w\times\w$ are isolated, while a neighborhood base at $\infty$ is formed by the sets
\[
U_{C,n,\varphi}=\{\infty\}\cup\{(\alpha,k,m)\in \kappa\times\w\times\w:\alpha\in \kappa\setminus C,\;k\geq n,\;m\ge\varphi(\alpha,k)\},
\]
where $C\subset \kappa$ is a subset of cardinality $|C|<\kappa$, $n\in\w$ and $\varphi:\kappa\times \w\to\w$ is a function. It is easy to see that the space $\Omega$ does not contain infinite compact subsets. Consequently, each network in $\Omega$ is a $k$-network. Since $\Omega$ has a unique non-isolated point, it is paracompact.

For every $n,m\in\w$ consider the following discrete families in $\Omega$
\[
\mathcal{N}_n=\big\{\{(\alpha,k,m)\}:\alpha\in\Omega,\;k\le n,\;m\in\w\}
\]
and
\[
\mathcal{N}_{n,m}=\big\{\{(\alpha,k)\}\times [m,\w):\alpha\in\Omega,\;k\le n\}.
\]
 Then the family
\[
\mathcal{N}=\big\{\{\infty\}\big\}\cup\bigcup_{n\in\w}\mathcal{N}_n\cup\bigcup_{n,m\in\w}\mathcal{N}_{n,m}
\]
is a $\sigma$-discrete network in $\Omega$; so $\Omega$ is an $\aleph$-space.

{\em Claim 1.} We claim that the family $\mathcal{N}$ is a Pytkeev network for $\Omega$, i.e., for every open set $U\subset \Omega$ and a set $A$ accumulating at a point $x\in U$ there is a set $N\in\mathcal{N}$ such that $N\subset U$ and $N\cap A$ is infinite.

Since all points of $\Omega$ except for $\infty$ are isolated, it suffices to check that $\mathcal{N}$ is a Pytkeev network at $\infty$.
Fix a neighborhood $U\subset \Omega$ of $\infty$ and a set $A\subset \Omega$ accumulating at $\infty$. Without loss of generality we can assume that the neighborhood $U$ is of basic form $U=U_{C,n,\varphi}$ for some set $C\subset\kappa$ with $|C|<\kappa$, $n\in\w$ and $\varphi:\kappa\times\w\to\w$. We claim that for some pair $(\alpha,k)\in(\kappa\setminus C)\times [n,\w)$ the intersection $A\cap (\{(\alpha,k)\}\times\w)$ is infinite. Indeed, otherwise, we could find a function $\psi:\kappa\times \w\to\w$ such that $\psi\ge\varphi$ and $A\cap (\{(\alpha,k)\}\times\w)\subset \{(\alpha,k)\}\times\big[0,\psi(\alpha,k)\big)$ for every $(\alpha,k)\in(\kappa\setminus C)\times [n,\w)$. Then $O_{C,n,\psi}$ is a neighborhood of $\infty$, disjoint with the set $A\setminus\{\infty\}$, which means that $A$ does not accumulate at $\infty$. This contradiction shows that for some $(\alpha,k)\in(\kappa\setminus C)\times[n,\w)$ the set $A$ has infinite intersection with the set $\{(\alpha,k)\}\times \w$ and hence with the set $N=\{(\alpha,k)\}\cup[\varphi(\alpha,k),\w)$. Taking into account that $N\in\mathcal{N}$ and $N\subset U_{C,n,\varphi}$, we conclude that $\mathcal{N}$ is a Pytkeev network at $\infty$.

{\em Claim 2.} We claim that $cn_\chi(\Omega,\infty)=\kappa$. The inequality $cn_\chi(\Omega,\infty)\leq\kappa$ follows from the fact that the family $\{\{\infty\}\cup \{x\}:x\in \Omega\}$ is a $cn$-network of cardinality $\kappa$ at $\infty$. Assuming that $cn_\chi(\Omega,\infty)<\kappa$, fix a $cn$-network $\mathcal{C}$ at $\infty$ of cardinality $|\mathcal{C}|=cn_\chi(\Omega,\infty)<\kappa$. We can suppose that $|C|>1$ and $\infty\in C$ for every $C\in \mathcal{C}$. In each  set $C\in\mathcal{C}$ fix a point $(\alpha_C,y_C,z_C)\in C\cap (\kappa\times\w\times\w)$ and consider the set $F=\{\alpha_C:C\in\mathcal{C}\}$. Then  the set
\[
U_\infty=\{\infty\}\cup\big((\kappa\setminus F)\times\w\times\w\big)
 \]
is a neighborhood of $\infty$ which does not contain a set $C\in\mathcal{C}$. So
\[
\bigcup\{C\in\mathcal{C}:\infty\in C\subset U_\infty\}=\emptyset
\]
is not a neighborhood of $\infty$ and hence $\mathcal{C}$ fails to be a $cn$-network at $\infty$.

{\em Claim 3.} We claim that the tightness $t(\Omega,\infty)$ of the space $\Omega$ at $\infty$ is $\kappa$; hence  the tightness $t(\Omega)$ of $\Omega$ is $\kappa$.

Clearly, $t(\Omega,\infty)\leq t(\Omega)\leq |\Omega|\leq \kappa$. Let us show that $t(\Omega,\infty)\geq\kappa$. Consider the set $\A := U_{C,n,\varphi}\setminus \{\infty\}$ for some $C\subset \kappa$ with $|C|<\kappa$, $n\in\w$ and  a function $\varphi:\kappa\times \w\to\w$. It is clear that $\infty\in \overline{\A}$. Fix arbitrarily a subset $A$ of $\A$ with $|A|<\kappa$. Denote by $B$ the projection of $A$ to $\kappa$; so $|B|<\kappa$. Then, by construction, the open neighborhood $U_{B,n,\varphi}$ of $\infty$ does not intersect with $A$. Thus $t(\Omega,\infty)\geq\kappa$.

Finally,  $\Omega$ fails to be a strict $\sigma$-space by  Corollary \ref{c-Aleph} and Proposition \ref{pTigh}.
\end{proof}

\begin{remark}{\em
Corollary \ref{c-Aleph} shows that the  ({\em global}) concept of the  $\nn\mbox{-}\sigma$-spaces is well-consistent with the ({\em local}) concept of the   countable $\nn$-character. On the other hand, as Example \ref{exa-A-non-CS} shows,  the \emph{(global)} and \emph{(local)} network properties dealing with  the $\sigma$-locally finite Pytkeev networks  and $\aleph$-spaces, respectively,  actually are not consistent. Since $\aleph$-spaces coincide with $cs^\ast\mbox{-}\sigma$-spaces by Theorem \ref{tA-cs}, $\aleph$-spaces can be naturally consistent {\em only} with the \emph{local property} to have countable $cs^\ast$-character, but not with countable $k$- or $ck$-character. This may suggest that the class of strict $\aleph$-spaces is  actually a more appropriate generalization of $\aleph_0$-spaces than the  class of $\aleph$-spaces. }
\end{remark}

Relations between various types of $\aleph$-spaces are given below.
\begin{proposition} \label{c--PytAleph}
Each  $\Pp$-space is a strict $\aleph$-space, and each strict  $\aleph$-space is an $\aleph$-space.
\end{proposition}
\begin{proof}
Let $X$ be a $\Pp$-space with a $\sigma$-locally finite $cp$-network $\DD$. For every $x\in X$, the proof of Theorem \ref{t-Lind-Set} shows that the family $\Nn(x):=\{ N\in\DD : x\in N\}$ is a countable $cp$-network at $x$. Proposition \ref{fB} implies that $\Nn(x)$ is also a $ck$-network at $x$. Thus $\DD$ is a $\sigma$-locally finite $ck$-network for $X$, i.e. $X$ is a strict $\aleph$-space.

Let $X$ be a  strict $\aleph$-space with a $\sigma$-locally finite $ck$-network $\DD$. As $X$ is regular, Remark \ref{r1} shows that $\DD$ is also a $k$-network for $X$.  Thus $X$ is an $\aleph$-space.
\end{proof}

In \cite{GK2} it is shown that a  topological space $X$ is cosmic (resp. an $\aleph_0$-space) if and only if $X$ is a  Lindel\"{o}f $\sigma$-space (resp. a  Lindel\"{o}f $\aleph$-space). Analogously we prove the following
\begin{proposition} \label{pLindel}
Let $X$ be a topological space. Then
\begin{enumerate}
\item[{\rm (i)}] $X$ is cosmic  if and only if $X$ is a  Lindel\"{o}f strict $\sigma$-space;
\item[{\rm (ii)}]  $X$ is  an $\aleph_0$-space if and only if $X$ is a  Lindel\"{o}f strict $\aleph$-space;
\item[{\rm (iii)}]  $X$ is  a $\Pp_0$-space if and only if $X$ is a  Lindel\"{o}f  $\Pp$-space.
\end{enumerate}
\end{proposition}
\begin{proof}
We prove only (iii). Assume that $X$ is a Lindel\"of  $\Pp$-space  with a $\sigma$-locally finite $cp$-network $\DD=\bigcup_{n\in\NN} \DD_n$. It is enough to prove that every $\DD_n$ is countable. For every $x\in X$ choose an open neighborhood $U_x$ of $x$ such that $U_x$ intersects with a finite subfamily $T(x)$ of $\DD_n$. Since $X$ is a Lindel\"{o}f space, we can find a countable set $\{ x_k\}_{k\in\NN}$ in $X$ such that $X =\bigcup_{k\in\NN} U_{x_k}$. Hence any $D\in\DD_n$ intersects with some $U_{x_k}$ and therefore $D\in T(x_k)$. Thus $\DD_n = \bigcup_{k\in\NN} T(x_k)$ is countable.
Conversely, if $X$ is a $\Pp_0$-space, then $X$ is Lindel\"{o}f (see \cite{Mich}) and it is trivially a $\Pp$-space.
\end{proof}

\begin{remark}{\em
If $X$ is a cosmic non-$\aleph_0$-space, then $X$ is not a strict $\aleph$-space by Proposition \ref{pLindel}.}
\end{remark}


\section{Metrizability conditions for topological spaces} \label{secMetr}


In this section we present some criterions for metrizability of topological spaces.

Following \cite[II.2]{Arhangel} we say  that a topological space $X$ has {\it countable fan tightness at a point } $x\in X$ if for each sets $A_n\subset X$, $n\in\NN$, with $x\in \bigcap_{n\in\NN} \overline{A_n}$ there are finite sets $F_n\subset A_n$, $n\in\NN$, such that $x\in \overline{\cup_{n\in\NN} F_n}$.
A space $X$ has {\it countable fan tightness} if it has  countable fan tightness at each point $x\in X$.

Recall that a topological space $X$ has the {\it property $\left( \alpha_{4}\right) $ at a point $x\in X$} if for any $\{x_{m,n}:\left( m,n\right) \in \mathbb{N}\times \mathbb{N}\}\subset X$ with $\lim_{n}x_{m,n}=x\in X$, $m\in \mathbb{N}$, there exists a sequence $\left( m_{k}\right) _{k}$ of distinct natural numbers and a sequence $\left( n_{k}\right) _{k}$ of natural numbers such that $\lim_{k}x_{m_{k},n_{k}}=x$; $X$ has the {\it property $\left( \alpha_{4}\right) $} or is an {\it $\left( \alpha_{4}\right) $-space} if it has the property $\left( \alpha_{4}\right)$ at each point $x\in X$. Nyikos proved in \cite[Theorem 4]{nyikos} that any Fr\'{e}chet-Urysohn topological group satisfies $\left( \alpha_{4}\right)$. However there are Fr\'{e}chet-Urysohn topological spaces which do not have $\left( \alpha_{4}\right)$.

Next proposition  recalls some criterions for a topological space to be first countable at a point. Note that (i)$\Leftrightarrow$(ii) is proved in \cite{Banakh1} and (i)$\Leftrightarrow$(iii) follows from Proposition 6 and Lemma 7 of \cite{BZ}.
\begin{proposition} \label{FirstCount}
Let $x$ be a point of a topological space $X$. Then the following assertions are equivalent:
\begin{itemize}
\item[{\rm (i)}] $X$  is first countable at $x$.
\item[{\rm (ii)}] $X$ has a countable Pytkeev network at $x$ and countable fan tightness at $x$.
\item[{\rm (iii)}] $X$ has a countable $cs^\ast$-network at $x$ and is a Fr\'{e}chet-Urysohn $\left( \alpha_{4}\right)$-space.
\end{itemize}
\end{proposition}

Recall also (see \cite{Mich}) that a point $x$ in a topological space $X$ is called an {\it $r$-point} if there is a sequence $\{ U_n\}_{n\in\NN}$ of neighborhoods of $x$ such that if $x_n\in U_n$, then $\{ x_n\}_{n\in\NN}$ has compact closure; we call $X$ to be an {\it $r$-space} if all of its points are $r$-points.
\begin{remark} \label{remMet} {\em
The first countable spaces and the locally compact spaces are  trivially $r$-spaces. So the Bohr compactification $b\mathbb{Z}$ of $\mathbb{Z}$ is an $r$-space, but since $b\mathbb{Z}$  has uncountable tightness it does not have countable fan tightness. On the other hand, there are spaces with countable fan tightness  which are not $r$-spaces. Indeed, let $X=C_p[0,1]$. Then $X$ has countable fan tightness by \cite[II.2.12]{Arhangel}. As $X$ has a neighborhood base at zero determined by finite families of points in $[0,1]$, for each sequence $\{ U_n\}_{n\in\NN}$ of neighborhoods of zero  we can find $z\in [0,1]$ and  $f_n \in U_n$ such that $f_n (z) \to \infty$. Hence the closure of $\{ f_n\}_{n\in\NN}$ is non-compact. Thus $X$ is not an $r$-space. We do not know non-metrizable $r$-spaces which have countable fan tightness.}
\end{remark}

Following Morita \cite{Morita}, a topological space $X$ is called an {\it $M$-space} if there exists a sequence $\{ \mathcal{U}_n\}_{n\in\NN}$ of open covers of $X$ such that: (i) if $x_n\in \bigcup \{ U\in \mathcal{U}_n : x\in U\}$ for each $n$, then $\{ x_n \}_{n\in\NN}$ has a cluster point; (ii) for each $n$, $\mathcal{U}_{n+1}$ star refines $\mathcal{U}_n$. Any countably compact space $X$ is an $M$-space; just set $\UU_n =\{ X\}$ for every $n\in\NN$. Countably compact subsets of an $\nn-\sigma$-space are metrizable as the next theorem shows.
\begin{theorem} \label{pComMet}
For a topological space $X$ the following assertions are equivalent:
\begin{itemize}
\item[{\rm (i)}] $X$ is metrizable.
\item[{\rm (ii)}] {\rm (\cite{OMe})} $X$ is an $\aleph$-space and an $r$-space.
\item[{\rm (iii)}]  {\rm (\cite{Morita})} $X$ is a $\sigma$-space and an $M$-space.
\end{itemize}
\end{theorem}
Consequently, each compact subset of a $\sigma$-space (in particular, a $cn\mbox{-}\sigma$-space) $X$ is metrizable. For locally compact spaces we have the following
\begin{proposition}
A locally compact space $X$ is metrizable if and  only if $X$ is a paracompact $\sigma$-space.
\end{proposition}
\begin{proof}
If $X$ is metrizable, then $X$ is paracompact by Stone's theorem \cite[5.1.3]{Eng} and is a $\sigma$-space by Proposition \ref{pMetr-Pyt}. Assume that $X$ is a paracompact $\sigma$-space. Then Theorem 5.1.27 of \cite{Eng} implies that $X=\oplus_{i\in I} X_i$, where $X_i$ is a clopen Lindel\"{o}f subset of $X$ for every $i\in I$. By \cite{GK2}, 
any $X_i$ is a locally compact cosmic space. Hence $X_i$ is a separable metrizable space by \cite[3.3.5]{Eng}. Thus $X$ is also metrizable.
\end{proof}

\textcolor{blue}{Let us note that in the following theorem} the implication  (i)$\Leftrightarrow$(ii) below  generalizes \cite[Theorem 1.9]{Banakh}, the implication (i)$\Leftrightarrow$(iii) is proved in \cite{GK2} and (i)$\Leftrightarrow$(v) follows from \cite{Morita}.
\begin{theorem} \label{tMetr-Pyt}
For a  topological space $X$ the following assertions are equivalent:
\begin{itemize}
\item[{\rm (i)}] $X$ is metrizable.
\item[{\rm (ii)}] $X$ is a  $\Pp$-space and has countable fan tightness.
\item[{\rm (iii)}] $X$ is an $\aleph$-space and is a Fr\'{e}chet-Urysohn  $\left( \alpha_{4}\right)$-space.
\item[{\rm (iv)}] $X$ is a strict $\aleph$-space and is a Fr\'{e}chet-Urysohn  $\left( \alpha_{4}\right)$-space.
\item[{\rm (v)}] $X$ is a strict $\sigma$-space and is an $M$-space.
\end{itemize}
\end{theorem}

\begin{proof}
(i)$\Rightarrow$(iv) Let $X$ be metrizable. Then clearly, $X$ is a Fr\'{e}chet-Urysohn $\left( \alpha_{4}\right) $-space, and Proposition \ref{pMetr-Pyt} implies that $X$ is a strict $\aleph$-space. (iv)$\Rightarrow$(iii) is trivial.

(iii)$\Rightarrow$(ii) By Corollary \ref{cAleph-cs}, $X$ has countable $cs^\ast$-character. Now the hypothesis and Proposition \ref{FirstCount} imply that $X$ is a $\Pp$-space and has countable fan tightness.

(ii)$\Rightarrow$(i) By Corollary \ref{c-Aleph}, we have $cp_\chi(X)\leq\aleph_0$. So $X$ is first countable by Proposition \ref{FirstCount}. Now  Proposition \ref{c--PytAleph} and Theorem \ref{pComMet}(ii) imply metrizability of $X$.

Since any metrizable space is a strict $\sigma$-space and each strict $\sigma$-space is a $\sigma$-space, the equivalence (i)$\Leftrightarrow$(v) immediately follows from Theorem \ref{pComMet}(iii) and Proposition \ref{pMetr-Pyt}.
\end{proof}

Since any cosmic space is separable, the next corollary contains \cite[Theorem 1.9]{Banakh}:
\begin{corollary} \label{cSepMet}
A topological space $X$ is second countable if and only if  $X$ is a $\Pp_0$-space and has countable fan tightness if and only if $X$ is an $\aleph_0$-space and is a Fr\'{e}chet-Urysohn $\left( \alpha_{4}\right) $-space.
\end{corollary}

\begin{example} \label{exaBohr}{\em
Let $\mathbb{Z}^+$ be the (discrete) group of integers $\mathbb{Z}$ endowed with the Bohr topology and $b\mathbb{Z}$ be the completion of $\mathbb{Z}^+$. It is easy to see that $\mathbb{Z}^+$ is  an $\aleph_0$-space. So
\[
1= cs^\ast_\chi(\mathbb{Z}^+) < ck_\chi(\mathbb{Z}^+) =cn_\chi(\mathbb{Z}^+)=\aleph_0 < \mathfrak{c}=\chi(\mathbb{Z}^+).
\]
Since every convergent sequence in $\mathbb{Z}^+$ is trivial, the family $\{ \{ 0\}\}$ is a $sc^\ast$-network at $0$ which is not a $cn$-network at zero. As each compact subset of $\mathbb{Z}^+$ is finite, the precompact group $\mathbb{Z}^+$ is not a $k$-space. The compact non-metrizable group $b\mathbb{Z}$ is a  dyadic compactum by Ivanovskij-Kuz'minov's theorem. Thus $b\mathbb{Z}$ has uncountable $cs^\ast$-character by \cite[Proposition 7]{BZ} and has uncountable $cn$-character by \cite{GK-GMS2}. } 
\end{example}

It is natural to ask whether  $\mathbb{Z}^+$ is a  $\Pp_0$-space.
Answering this  question Banakh \cite{Banakh} proved the following: A precompact group has the strong Pytkeev property if and only if it is metrizable. So Theorem \ref{tMetPrecomGr} generalizes the latter mentioned theorem to locally precompact groups although provides  a similar proof.  Recall that, a subset $E$ of a topological group $G$ is called {\it left-precompact} (respectively, {\it right-precompact, precompact}) if, for every open neighborhood $U$ of the unit $e\in G$, there exists a finite subset $F$ of $G$ such that $E\subseteq F\cdot U$ (respectively, $E\subseteq U\cdot F$, $E\subseteq F\cdot U$ and $E\subseteq U\cdot F$). If $E$ is symmetric the three different definitions coincide. A topological group $G$ is called {\it locally precompact} if it has a base at the unit consisting of symmetric precompact sets (or in other words, $G$ embeds into a locally compact group).

\begin{proof}[Proof of Theorem \ref{tMetPrecomGr}]
The implication (i)$\Rightarrow$(ii) follows from Proposition \ref{pMetr-Pyt}, and (ii)$\Rightarrow$(iii) follows from Corollary \ref{c-Aleph}. Let us prove (iii)$\Rightarrow$(i). Assume that $G$ has the strong Pytkeev property with a closed countable $cp$-network $\mathcal N$ at the unit $e$ of $G$. Suppose for a contradiction that $G$ is not first countable. Then there exists a symmetric precompact neighborhood $U$ of $e$ such that if $N\in \mathcal{N}$ and $N\subset U$ then $N$ is nowhere dense in $G$. Indeed, otherwise we can take a neighborhood  $V\subset G$ of  $e$ with $V^{-1} V\subset U$ and find $N\subset V$ such that $N^{-1} N \subset U$ contains a neighborhood of $e$, that means that the countable family $\{ N^{-1} N: N\in \mathcal{N}\}$ is a base at $e$. Let $\{N_k'\}_{k\in\NN}$ be an enumeration of the family
\[
\mathcal{N}'=\{N\in \mathcal{N}: N \mbox{ is nowhere dense in } G  \mbox{ and } N\subset U\},
\]
which is a closed $cp$-network at $e$. Fix arbitrarily $a_1 \not\in N'_1$.  Using induction and the nowhere density of the sets $N_k'$ one can construct a sequence $(a_k)_{k\in\NN} \subset U$ of distinct points of $G$ such that $a_n\notin\bigcup_{k<n}\bigcup_{m\le n}a_kN_m'$. Consider the set $A=\{a_k^{-1}a_n:k<n\}$. We claim that this set contains $e$ in its closure. Indeed, for every neighborhood $V$ of $e$, we can find a neighborhood $W$ of $e$ such that $W^{-1}W\subset V\cap U$, and using the precompactness of $U$ one can find a finite subset $F\subset G$ such that $U\subset FW$. By the Pigeonhole Principle, there are two numbers $k<n$ such that $a_k,a_n\in xW$ for some $x\in F$. Then $a_k^{-1}a_n\in W^{-1}W\subset V$, and hence $A\cap V\ne\emptyset$. Since $\mathcal N'$ is a $cp$-network at $e$, there is a number $q\in\NN$ such that the set $B:= N_q'\cap A$ is infinite. But this is not possible because $ B \subset\{a_i^{-1}a_j:i<j<q\}$ and hence $B$ is finite.
\end{proof}

Consequently, this shows that  the group  $\mathbb{Z}^+$ does not have the strong Pytkeev property.


\section{$\nn$-networks and operations over topological spaces} \label{secPro}


In this section we consider some standard operations in the class of spaces with countable $\nn$-character.
For $\sigma$-, $\aleph$-, $\aleph_0$- and   $\Pp_0$-spaces as well as for spaces with countable $cs^\ast$-character all the following results are well-known (see \cite{Banakh,BZ,gruenhage,Mich}).

Next two obvious  propositions show that the classes of topological spaces with countable character of various types are closed under taking subspaces and topological sums.
\begin{proposition} \label{pSub}
For $\nn\in\mathfrak{N}$, if $\Nn$ is an [$\sigma$-locally finite] $\mathfrak{n}$-network (at a point $x$) in a topological space $X$, then for every subspace $A\subset X$ (such that $x\in A$) the family $\Nn |_A := \{ N\cap A: n\in\Nn\}$ is an  [$\sigma$-locally finite]  $\mathfrak{n}$-network (at the point $x$) in the space $A$.
\end{proposition}

\begin{proposition} \label{pSum}
For $\nn\in\mathfrak{N}$, if $\Nn_i$ is an  [$\sigma$-locally finite] $\mathfrak{n}$-network in a topological space $X_i$, $i\in I$, then the family $\Nn = \bigcup_{i\in I} \Nn_i$ is an [$\sigma$-locally finite]  $\mathfrak{n}$-network in the topological sum $\oplus_{i\in I} X_i$.
\end{proposition}

If $X=\prod_{n\in\NN} X_n$ and $n\in\NN$, we denote by  $p_n$ and $\pi_n$ the projections of $X$ onto $X_n$ and $X_1\times\cdots\times X_n$ respectively. For countable (Tychonoff) product we have the following.
\begin{proposition} \label{pProd}
For $\nn\in\mathfrak{N}$, if $\Nn_i =\{ N^i_n\}_{n\in\NN}$ is a countable $\mathfrak{n}$-network at a point $x_i$ of a topological space $X_i$, $i\in\NN$, then the countable family
\[
\Nn :=\left\{ N^1_{m_1}\times \cdots \times N^n_{m_n}\times \prod_{k>n} X_k : n\in\NN, x_i\in N^i_k \in \Nn_i \right\}
\]
is an $\mathfrak{n}$-network at the point $x=(x_i)$ of $X=\prod_{i\in\NN} X_i$.
\end{proposition}

\begin{proof}
Let $U= \prod_{i=1}^n U_i \times\prod_{i>n} X_i$ be a neighborhood of $x$, where $U_i$ is a neighborhood of $x_i$ for all $1\leq i\leq n$. 

(1) Assume that $\Nn_i$ are $cn$-networks.  By definition, for every $1\leq i\leq n$, the set $W_i := \bigcup \{ N^i_k \in\Nn_i : x_i\in N^i_k \subseteq U_i \}$ is  a neighborhood of $x_i$. Clearly,
\[
\prod_{i=1}^n W_i \times\prod_{i>n} X_i \subseteq \bigcup \left\{ \prod_{i=1}^n N^i_{l_i} \times\prod_{i>n} X_i \in\Nn : \ x_i\in N^i_{l_i} \subseteq U_i, 1\leq i\leq n \right\}.
\]
Thus $\Nn$ is a countable $cn$-network in $X$ at $x$.

The case when  $\Nn_i$ are networks is considered analogously.

(2) Assume that $\Nn_i$ are $ck$-networks.  By definition, for every $1\leq i\leq n$, there exists a neighborhood $W_i \subset U_i$ of $x_i$ such that for each compact subset $K_i$ of $W_i$ there is a finite subfamily $\mathcal{F}_i \subset \Nn_i$ such that $x_i \in \bigcap \mathcal{F}_i$ and $K_i \subset \bigcup \mathcal{F}_i \subset U_i$. Set $W:= \prod_{i=1}^n W_i \times\prod_{i>n} X_i$. Then each compact subset $K$ of $W$ is contained in a set of the form $\prod_{i=1}^n K_i \times\prod_{i>n} X_i$, where $K_i$ is a compact subset of $W_i$. Clearly,
\[
K \subset \bigcup \left\{ \prod_{i=1}^n N^i_{l_i} \times\prod_{i>n} X_i \in\Nn : \ x_i\in N^i_{l_i}\in \mathcal{F}_i, 1\leq i\leq n \right\} \subset U.
\]
Thus $\Nn$ is a countable $ck$-network at $x$.

The cases $\Nn_i$ are $cs^\ast$-networks or $cs$-networks are considered analogously.

(3) Assume that $\Nn_i$ are $cp$-networks. First we prove the following claim.

\vspace{2mm}

{\bf Claim}. The product $X_1\times X_2$ of two  spaces $X_1$ and $X_2$ with countable $cp$-networks at $x_1$ and $x_2$ has a countable $cp$-network at $x=(x_1,x_2)$.

\begin{proof}
Let $A$ be a subset of $X_1\times X_2$ such that $x\in \overline{A}\setminus A$ and $U_1\times U_2$ be a neighborhood of $x$. 

For $i\in\{ 1,2\}$, consider the countable family $\mathcal{P}_i =\{ N^i_n \in \Nn_i : x_i \in N^i_n \subset U_i\}$ and let $\mathcal{P}_i =\{ N^i_{n_k} \}_{k\in\omega}$ be its enumeration. For every $k\in\omega$, set $P_{k,i} =\bigcup_{l\leq k} N^i_{n_l}$ and $P_k = P_{k,1}\times P_{k,2}$. Then $x\in\bigcap_{k\in\omega} P_k \subset \bigcup_{k\in\omega} P_k \subset U_1\times U_2$. By (1) the family $\mathcal{N}$ is a $cn$-network at $x$, so $V:= \bigcup_{k\in\omega} P_k$ is a neighborhood of $x$.
We  show that $A_k :=P_k\cap A$ is infinite for some $k\in\omega$.

Suppose by  a contradiction that $A_k$ is finite for all $k\in\omega$. Then for every $a=(a_1, a_2)\in (V\cap A)\setminus A_0$ we can find a unique number $k_a\in\omega$ such that $a\in A_{k_a +1}\setminus A_{k_a}$. Since $a\not\in P_{k_a}$, we fix arbitrarily $n_a\in \{ 1,2\}$ such that $a_{n_a} \not\in P_{k_a, n_a}$.

For $n\in\{ 1,2\}$, set $A(n):= \{ a=(a_1, a_2)\in (V\cap A)\setminus A_0 : n_a =n\}$ and $B_n:= p_n(A(n)) \subset X_n$. We claim that $x_n \not\in \overline{B_n}$. We show first that $x_n \not\in B_n$. Indeed, assuming that $x_n \in B_n$ we can find $a=(a_1, a_2)\in A(n)$ such that $a_n = x_n$. However, by the definition of $A(n)$, we have $a_{n} \not\in P_{k_a,n}$ and hence $a_n \not= x_n$. This contradiction shows that $x_n \not\in B_n$. Now we suppose for a contradiction that $x_n \in \overline{B_n}$. Since $\mathcal{N}_n$ is a $cp$-network at $x_n$, we can find $P_{k,n}\in\mathcal{N}_n$ such that $P_{k,n}\cap B_n$ is infinite. On the other hand, for every $a=(a_1, a_2)\in A(n)\setminus A_k$ we have $a_n \not\in P_{k,n}$. So the intersection $P_{k,n}\cap B_n$ is contained in $A_k$ and hence it is finite. This contradiction shows that $x_n \not\in \overline{B_n}$.

For $n\in\{ 1,2\}$, choose an open neighborhood $W_n$ of $x_n$ such that $W_n \cap \overline{B_n}=\emptyset$. Then
\[
[(V\cap A)\setminus A_0] \cap (W_1\times W_2) = (A(1)\cup A(2)) \cap (W_1\times W_2) = \emptyset,
\]
and hence $x\not\in \overline{A}$, a contradiction. Thus $A_k :=P_k\cap A$ is infinite for some $k\in\omega$.

Since $P_k$ is a finite union of elements from $\mathcal{N}$, we obtain that $\mathcal{N}$ is a $cp$-network at $x$. The claim is proved.
\end{proof}

Now let $A$ be a subset of $X$ such that $x\in \overline{A}\setminus A$. 
Set $C:=\pi_n (A)\setminus \{ \pi_n (x)\}$ and $D:= \pi_n^{-1} (\pi_n (x))\cap A$. If $D$ is {\it infinite}, then
\[
D\subset \left( N^1_{m_1}\times \cdots \times N^n_{m_n}\times \prod_{i>n} X_i\right) \cap A,
\]
and hence the last intersection is infinite for any $N^i_k \in\Nn_i$, as desired. If $D$ is {\it finite}, then $\pi_n (x)\in \overline{C}\setminus C$. By Claim, there is $N:= N^1_{m_1}\times \cdots \times N^n_{m_n}$ such that $N\cap C$ is infinite. Then $(N\times \prod_{i>n} X_i) \cap A$ is infinite as well. Thus $\mathcal{N}$ is a countable $cp$-network at $x$.
\end{proof}

\begin{proposition} \label{p-AlepProd}
For $\nn\in\mathfrak{N}$, the countable product of $\nn\mbox{-}\sigma$-spaces is an $\nn\mbox{-}\sigma$-space.
\end{proposition}

\begin{proof}
Let $\DD_k =\bigcup_{s\in\NN} \DD_{s,k}$ be a closed $\sigma$-locally finite $\nn$-network in an $\nn\mbox{-}\sigma$-space $X_k, k\in\NN$. Then the space $X:=\prod_{k\in\NN} X_k$ is regular by \cite[2.3.11]{Eng}.
For each $s,n\in\NN$ set
\[
\Nn_{s,n} :=\left\{ N^1_{m_1}\times \cdots \times N^n_{m_n}\times \prod_{i>n} X_i : \;  N^k_i \in \DD_{s,k}, \ k\in\{ 1,\dots,n\} \right\}.
\]
Clearly, $\Nn_{s,n}$ is locally finite for every $s,n\in\NN$. Hence the family $\Nn := \bigcup_{s,n\in\NN} \Nn_{s,n}$ is $\sigma$-locally finite. Fix $x=(x_k)\in X$. For every $k\in\NN$, the family $\{ D\in\DD_k :x_k\in D\}$ is a countable $\nn$-network at $x_k$, see the proof of Theorem \ref{t-Lind-Set}. Now Proposition \ref{pProd} shows that $\Nn$ is an $\nn$-network for $X$.
\end{proof}

Propositions \ref{pSub}--\ref{p-AlepProd} imply

\begin{corollary} \label{c-Oper}
For $\nn\in\mathfrak{N}$, the class of topological space with countable $\nn$-character is closed under taking subspaces, topological sums and countable products.
\end{corollary}

\begin{corollary} \label{c-Oper-Sigma}
For $\nn\in\mathfrak{N}$, the class of $\nn\mbox{-}\sigma$-spaces is closed under taking subspaces, topological sums and countable products.
\end{corollary}

\begin{corollary}[\cite{Banakh,Mich}] \label{cProperty}
The classes of cosmic, $\aleph_0$-spaces and  $\Pp_0$-spaces are closed under taking subspaces, countable topological sums and countable products.
\end{corollary}


\section{Function spaces} \label{secFunc}


The following theorem initialized by Michael's Theorem \ref{tMichael}(i) is one of the most important and interesting applications of $\aleph$-spaces.
\begin{theorem} \label{tMichael}
Let $X$ be an $\aleph_0$-space. Then:
\begin{itemize}
\item[{\rm (i)}] {\em (\cite{Mich})} If $Y$ is an  $\aleph_0$-space, then $C_c(X,Y)$ is also an $\aleph_0$-space.
\item[{\rm (ii)}] {\em (\cite{foged,OMe})} If $Y$ is a (paracompact) $\aleph$-space, then  $C_c(X,Y)$ is  a (paracompact) $\aleph$-space.
\end{itemize}
\end{theorem}

Recall that for a first countable space $X$ the space  $C_c(X)$ is metrizable if and only if it is Fr\'{e}chet-Urysohn, see \cite{McCoy}. The following proposition says something similar for the case $C_c(X,G)$, where $G$ is  an arbitrary topological group.
\begin{proposition}\label{nice}
Let $X$ be an $\aleph_{0}$-space and $G$ a topological group. If $G$ is an $\aleph$-space, then  $C_{c}(X,G)$ has countable $cs^\ast$-character.  Consequently, $C_{c}(X,G)$ is metrizable if and only if it is Fr\'{e}chet-Urysohn.
\end{proposition}

\begin{proof}
 By Theorem \ref{tMichael} the space $C_{c}(X,G)$ is an $\aleph$-space. By Corollary \ref{cAleph-cs} the space $C_{c}(X,G)$ has countable $cs^\ast$-character. Assume that $C_c(X,G)$ is  Fr\'{e}chet-Urysohn. Since every Fr\'echet-Urysohn group which has countable $cs^\ast$-character is metrizable  by  \cite[Theorem 3]{BZ}, the  topological group $C_c(X,G)$ is metrizable.  The converse is trivial.
\end{proof}

Theorem \ref{tMichael} inspires the following question.
\begin{question} \label{question-Func}
Let $X$ be an $\aleph_0$-space and $Y$ be a (paracompact) strict $\aleph$-space or a (paracompact) $\Pp$-space. Is $C_c(X,Y)$ a (paracompact) strict $\aleph$-space or a (paracompact) $\Pp$-space, respectively?
\end{question}
Note that  $C_c(X,Y)$ is a paracompact $\Pp$-space for any $\aleph_0$-space $X$ and each {\em metrizable} space $Y$ (see \cite{BG}). We know (see Corollary \ref{c-Aleph}) that any  $\Pp$-space has the strong Pytkeev property. So these results and  Theorem \ref{t-strong-Pyt-A} increase a hope that the answer to Question \ref{question-Func} might be positive.

In order to prove Theorem \ref{t-strong-Pyt-A} we  need the following lemma proved in \cite[Theorem 5, p. 223]{Kelley}.
\begin{lemma} \label{l-Kelley}
Let $C$ be a compact subspace of a Hausdorff space  $X$. Then the map $(x,f)\mapsto f(x)$, from $C\times C_c(X,Y)$ to $Y$, is continuous.
\end{lemma}

\textcolor{blue}{
Let $X$ and $Y$ be topological spaces. For arbitrary subsets $A\subseteq X$ and $B\subseteq Y$ set
\[
[A;B]:= \{ f\in C(X,Y): \ f(A)\subseteq B\}.
\]
The compact-open topology on the space $C(X,Y)$ is defined by a subbase consisting of the sets $[K;U]$ with $K\subseteq X$ compact and $U\subseteq Y$ open.
}

We are at the position to prove Theorem \ref{t-strong-Pyt-A}.

\begin{proof}[Proof of Theorem \ref{t-strong-Pyt-A}]
For the $\aleph_0$-space $X$ fix a countable $k$-network  $\KK$, which is closed under taking finite unions and finite intersections. For fixed  $\nn\in\{ck,cp\}$, fix a closed (increasing) $\sigma$-locally finite $\nn$-network  $\DD =\bigcup_{j\in\w} \DD_j$ in the  $\nn\mbox{-}\sigma$-space $Y$. To prove two cases (1)  and (2) of Theorem  \ref{t-strong-Pyt-A}  we need  to show that for every function $f\in C_c(X,Y)$ there exists a countable $\nn$-network at $f$.

Fix $f\in C_c(X,Y)$. Since $f(X)$ is a Lindel\"{o}f subspace of $Y$, Theorem \ref{t-Lind-Set} implies that  there exists a sequence $\DD_f =\{ D_k\}_{k\in\NN}\subset \DD$  which is a countable $\nn$-network at each point of $f(X)$ and satisfies the condition: if $\mathrm{K}\subset f(X)\cap U$ with $\mathrm{K}$ compact and $U$ open, then there is an open subset $W$ of $Y$ such that
\begin{itemize}
\item[{\rm (a)}] $\mathrm{K}\subset W \subset \bigcup_{k\in I(U)} D_k \subset U$, where $I(U)=\{ k\in\NN : D_k\subset U\}$, and
\item[{\rm (b)}]  for each compact subset $C$ of $W$ there is a finite subfamily $\alpha$ of $I(U)$ for which $C\subset \bigcup_{k\in\alpha} D_k$.
\end{itemize}
Let $\Nn_f$ be the countable family containing of all finite unions and intersections of elements of the sequence $\DD_f$. We claim that the countable family
\[
[\kern-2pt[ \KK; \Nn_f ]\kern-2pt]=\big\{[K_1;N_1]\cap \dots\cap [K_n;N_n]: K_1,\dots,K_n\in\KK,\;N_1,\dots,N_n\in\mathcal{N}_f \big\}
\]
is  an $\nn$-network at $f$ in $C_c(X,Y)$.

Fix an open neighborhood $O_f\subset C_c(X,Y)$ of $f$. Without loss of generality we can assume  that the neighborhood $O_f$ is of basic form
\[
O_f = [C_1;U_1]\cap\cdots\cap [C_n;U_n]
\]
for some compact sets $C_1,\dots,C_n$ in $X$ and some open sets $U_1,\dots,U_n$ in $Y$.

For every $i\in\{ 1,\dots, n\}$, consider the countable family
\[
\KK_i :=\{ K\in\KK : C_i \subseteq K\subseteq f^{-1}(U_i)\},
\]
and let $\KK_i = \{ K'_{i,j} \}_{j\in\w}$ be its enumeration. For every $j\in\NN$ we set $K_{i,j} := \bigcap_{k\leq j} K'_{i,j}$. It follows that the decreasing sequence $\{ K_{i,j}\}_{j\in\w}$ converges to $C_i$ in the sense that each open neighborhood of $C_i$ contains all but finitely many sets $K_{i,j}$.

For every $i\in\{ 1,\dots, n\}$, consider the countable family (which is non-empty by (b))
\[
\Nn_i :=\{ N\in\Nn_f : f(C_i) \subset N\subset U_i\},
\]
and let $W_i$ be an open neighborhood of $f(C_i)$  satisfying (a) and (b). Let $\{ N'_{i,j} \}_{j\in\w}$ be an enumeration of $\Nn_i$. For every $j\in\w$ we set $N_{i,j} := \bigcup_{k\leq j} N'_{i,j} \in \Nn_i$. It follows from (a) that $\{ N_{i,j}\}_{j\in\w}$ is an increasing sequence of sets in $Y$ with
\[
f(C_i)\subset  W_i \subset \bigcup_{j\in\w} N_{i,j} \subset U_i \; \mbox{ and } \; f(C_i)\subset \bigcap_{j\in\w} N_{i,j}.
\]

Then the sets
\[
\FF_j := \bigcap_{i=1}^n [K_{i,j};N_{i,j}] \in [\kern-2pt[ \KK; \Nn_f ]\kern-2pt],\;\;j\in\w,
\]
form an increasing sequence of sets in the function space $C_c(X,Y)$. Set $W_f :=\bigcap_{i=1}^n [C_i; W_i]$.

\begin{claim}\label{claim-1}
$f\in W_f =\bigcap_{i=1}^n [C_i; W_i]\subset \bigcup_{j\in\w} \FF_j \subset  O_f=\bigcap_{i=1}^n [C_i; U_i].$
\end{claim}

\begin{proof}
We need to prove only the first inclusion. Suppose for a contradiction that there exists a function $g\in  \bigcap_{i=1}^n [C_i; W_i]$ which does not belong to $\bigcup_{j\in\w} \FF_j$. Then for every $j\in\w$ we can find an index $i_j\in \{1,\dots,n\}$ such that $g\not\in [K_{i_j,j};N_{i_j,j}]$. This means that $g(x_j)\not\in N_{i_j,j}$ for some point $x_j\in K_{i_j,j}$. By the Pigeonhole Principle, there is $m\in \{1,\dots,n\}$ such that the set $J_m := \{ j\in\NN : i_j =m\}$ is infinite. As the decreasing sequence $\{ K_{m,j} \}_{j\in J_m}$ converges to the compact set $C_m$, the set $C_m \cup \{ x_j\}_{j\in J_m}$ is compact.

Since each compact subset of the $\aleph_0$-space $X$ is metrizable (see Theorem \ref{pComMet}), we can find an infinite subset $J'$ of $J_m$ such that the sequence  $\{ x_j\}_{j\in J'}$ converges to some point $x' \in C_m$. As $g$ is continuous, the sequence $\{ g(x_j)\}_{j\in J'}$ converges to the point $g(x') \in g(C_m) \subset W_m$,  and hence we can assume also that $g(x_j)\in W_m$ for every $j\in J'$. Then we can apply (b) for the compact set $C'=g(C_m) \cup \{ g(x_j)\}_{j\in J'}$ to find a finite subfamily $\FF$ of $\Nn_f$ such that $C'\subset \bigcup\FF$. Consequently, by construction, there is $N_{m,j_0}$ containing $C'$. But this contradicts the choice of the points $x_j$. This contradiction proves the inclusion $\bigcap_{i=1}^n [C_i; W_i]\subset \bigcup_{j\in\w}\FF_j$.
\end{proof}
By Claim \ref{claim-1}, without loss of generality we shall assume that $f\in \FF_j$ for every $j\in\w$.

We continue the proof by distinguishing  two cases which cover the proof of the theorem.

(1): {\em Assume that $\DD$ is a $cp$-network.} Given a subset $A\subset C_c(X,Y)$ with $f\in \overline{A}$ we need to find a set $\FF\in [\kern-2pt[ \KK; \Nn_f ]\kern-2pt]$ such that $f\in\FF \subset O_f$ and moreover $A\cap \FF$ is infinite if $f$ is an accumulation point of the set $A$. We can suppose that $A\subset W_f$.

If $f$ is an isolated point of $C_c(X,Y)$, then $f\in \FF_j \subset O_f$ for all $j\in\w$. Moreover, if $O_f=\{ f\}$, then $\{ f\}\in [\kern-2pt[ \KK; \Nn_f ]\kern-2pt]$, and we are done.

Assume now that $f$ is an accumulation point of $A$ in $C_c(X,Y)$. We show that $A\cap \FF_j$ is infinite for some $j\in\w$. Suppose for a contradiction that for every $j\in\w$ the intersection $A_j := \FF_j \cap A$ is finite. Then, by Claim \ref{claim-1},  $A=A\cap W_f = \bigcup_{j\in\w} A_j$ is the countable union of the increasing sequence $\{ A_j\}_{j\in\w}$ of finite subsets of $C_c(X,Y)$. Below we follow the proof of Theorem 2.1 of \cite{Banakh}.


For every function $\alpha\in A\setminus A_0$ we denote by $j_\alpha$ the unique natural number such that $\alpha\in A_{j_\alpha +1} \setminus A_{j_\alpha} = A_{j_\alpha +1} \setminus \FF_{j_\alpha}$. Since $\alpha\not\in \FF_{j_\alpha} =\bigcap_{i=1}^n [K_{i,j_\alpha};N_{i,j_\alpha}]$, fix arbitrarily an index $i_\alpha\in\{1,\dots,n\}$ such that $\alpha\not\in [K_{i_\alpha,j_\alpha};N_{i_\alpha,j_\alpha}]$ and a point $x_\alpha\in K_{i_\alpha,j_\alpha}$ such that $\alpha(x_\alpha) \not\in N_{i_\alpha,j_\alpha}$.

For every $i\in\{1,\dots,n\}$ consider the subsequence
\[
A(i):= \{ \alpha\in A\setminus A_0 : \ i_\alpha = i\}
\]
and observe that $A\setminus A_0 = \bigcup_{i=1}^n A(i)$.

For every $i\in\{1,\dots,n\}$, set $B_i :=\{ \alpha(x_\alpha): \alpha\in A(i)\}\subset Y$. We claim that the set $f(C_i)$ does not have accumulation points of $B_i$. Indeed, suppose for a contradiction that there  is a point $y\in f(C_i)$ which is an accumulation point of $B_i$. As $\Nn_f$ is a $cp$-network at $y$, there is $N\in\Nn_f$ such that $y\in N\subset W_i$ and $N\cap B_i$ is infinite. Since $\Nn_f$ is closed under taking finite unions, there is $N_{i,j}$ such that $f(C_i)\cup N\subset N_{i,j}\subset W_i$. But the choice of the points $x_\alpha$ guarantees that $\alpha(x_\alpha) \not\in N_{i,j}$ for all $\alpha\in A(i)\setminus A_j$, which yields that the intersection $B_i \cap N_{i,j}\subset \{ \alpha(x_\alpha): \alpha\in A(i)\cap A_j\}$ is finite. This contradicts the choice of the set $N\subset N_{i,j}$. Thus every point $y\in f(C_i)$ has an open neighborhood $O_y \subset W_i$ with finite $O_y\cap B_i$. Since $f(C_i)$ is compact we can find a finite family $Z_i \subset f(C_i)$ such that $V_i := \bigcup_{y\in Z_i} O_y \subset W_i$ is an open neighborhood of $f(C_i)$ having a finite intersection with the set $B_i$.

Since the decreasing sequence $\{ K_{i,j}\}_{j\in\w}$ converges to $C_i$, there is a number $j_i\in\w$ such that $K_{i,j_i} \subset f^{-1}(V_i)$. Take a sufficiently  large $j_i$ such that $V_i \cap \{ \alpha(x_\alpha): \alpha\in A(i)\setminus A_{j_i}\}=\emptyset$. Then the set
\[
C'_i := C_i \cup \{ x_\alpha : \alpha\in A(i)\setminus A_{j_i}\} \subset K_{i,j_i}
\]
is a compact subset of $f^{-1}(V_i)$, and hence the set $\widetilde{O}_f := \bigcap_{i=1}^n [C'_i;V_i]$ is an open neighborhood of $f$ in $C_c(X,Y)$. By construction, for every $1\leq i\leq n$ and each $\alpha\in A(i)\setminus A_{j_i}$, we have $x_\alpha\in C'_i$ and $\alpha(x_\alpha)\not\in V_i$; so $\alpha\not\in \widetilde{O}_f$. Hence
\[
A\cap \widetilde{O}_f=\left(\widetilde{O}_f\cap \bigcup_{i=1}^n A_{j_i}\right) \cup \left(\widetilde{O}_f\cap \bigcup_{i=1}^n  A(i)\setminus A_{j_i}\right) \subset \widetilde{O}_f\cap \bigcup_{i=1}^n A_{j_i}
\]
is finite and $f$ is not an accumulation point of $A$, a contradiction. Thus $A\cap \FF_j$ is infinite for some $j\in\w$. Therefore $[\kern-2pt[ \KK; \Nn_f ]\kern-2pt]$ is a countable $cp$-network at $f$.

(2): {\em Assume that  $\DD$ is a $ck$-network.} We have to show that the countable family $[\kern-2pt[ \KK; \Nn_f ]\kern-2pt]$ is  a $ck$-network at $f$. For this purpose it is enough to prove that the open neighborhood $W_f$ of $f$ witnesses $O_f$ in the definition of $ck$-network at $f$.

Fix a compact subset $A$ of $W_f$. We show that $A\subset \FF_j$ for some $j\in\w$. Suppose for a contradiction that $A\setminus \FF_j \not=\emptyset$ for every $j>0$. Observe that $C_c(X,Y)$ is an $\aleph$-space (see Theorem \ref{tMichael}(ii)), and hence $A$ is metrizable by Theorem \ref{pComMet}. As $A$ is compact, we can find a function $g_j\in A\setminus \FF_j$ such that the sequence $\{ g_j\}_{j\in\w}$ converges to some function $g_0\in A$ in $C_c(X,Y)$. So we can assume that $A=\{ g_j\}_{j\in\w}$.

For every $j>0$ we can find an index $i_j\in \{1,\dots,n\}$ such that $g_j\not\in [K_{i_j,j};N_{i_j,j}]$. This means that $g_j(x_j)\not\in N_{i_j,j}$ for some point $x_j\in K_{i_j,j}$. By the Pigeonhole Principle, there is $m\in \{1,\dots,n\}$ such that the set $J_m := \{ j\in\NN : i_j =m\}$ is infinite. As the decreasing sequence $\{ K_{m,j} \}_{j\in J_m}$ converges to the compact set $C_m$, the set $C_m \cup \{ x_j\}_{j\in J_m}$ is compact.
Since each compact subset of the $\aleph_0$-space $X$ is metrizable, we can find an infinite subset $J'$ of $J_m$ such that the sequence  $\{ x_j\}_{j\in J'}$ converges to some point $x_0 \in C_m$.

Observe that $g_0(x_0)$ belongs to the open set $W_m \subset Y$. Applying Lemma \ref{l-Kelley} with $C=\{ x_0\} \cup \{ x_j\}_{j\in J'}$ we can find an infinite subset $J''$ of $J'$ such that $g_k(x_j)\in W_m$ for every $j,k\in J''\cup\{ 0\}$. Now we apply once again  Lemma \ref{l-Kelley} to the compact set $C'=C_m \cup \{ x_j\}_{j\in J''}$ to obtain that the set
\[
T:=\left\{ g_j (x): \ x\in C', j\in J''\cup\{ 0\} \right\}
\]
is a compact subset of $Y$ contained in $W_m$. Then we  apply (b) for the compact set $T$ to find a finite subfamily $\FF$ of $\Nn_f$ such that $T\subset \bigcup\FF$. Consequently, by construction, there is $N_{m,j_0}$ containing $T$. But this contradicts the choice of the points $x_j$. This contradiction proves that $K\subset \FF_j$ for some $j\in\w$, and hence $[\kern-2pt[ \KK; \Nn_f ]\kern-2pt] $ is  a $ck$-network at $f$.
\end{proof}

If $Y$ is a $\Pp_0$-space, then the family $\DD_f$ in  the proof of Theorem \ref{t-strong-Pyt-A} can be chosen to be common for all $f\in C_c(X,Y)$, so we obtain the following remarkable results.
\begin{corollary}[\cite{Banakh}] \label{c-Banakh}
If $X$ is an $\aleph_0$-space and $Y$ is a  $\Pp_0$-space, then $C_c(X,Y)$ is a $\Pp_0$-space.
\end{corollary}
Let $X=\mathbb{Q}$ be the set of rational numbers. Then $C_c(\mathbb{Q})$ is a  $\Pp_0$-space by Corollary \ref{c-Banakh}; so the answer to Question 4 from \cite{GKL2} is negative  that was noticed by Banakh in \cite{Banakh}. The locally convex space $C_c(\mathbb{Q})$ gives also a negative answer to Questions 1 and 7 from \cite{GKL2}, see Remarks 3 and 6 in \cite{GKL2}.

Below we pose a few natural questions which are inspired by the corresponding results for $\aleph_0$-spaces and $\aleph$-spaces.

Recall that a map $f:Y\to X$ of topological spaces is {\it compact-covering} if each compact subset of $X$ is the image of a compact subset of $Y$.
Michael \cite{Mich} obtained the following characterizations of cosmic and $\aleph_0$-spaces.
\begin{theorem}[\cite{Mich}] \label{tCA}
Let $X$ be a regular space. Then:
\begin{itemize}
\item[{\rm (i)}] $X$ is cosmic if and only if $X$ is a  continuous image of a separable metric space.
\item[{\rm (ii)}] $X$ is an $\aleph_0$-space  if and only if $X$ is a compact-covering image of a separable metric space.
\end{itemize}
\end{theorem}

\begin{question}
Find a characterization of $\Pp_0$-spaces  analogously to the characterization of $\aleph_0$-spaces given in Theorem \ref{tCA}.
\end{question}

Recall that a mapping $f: X\to Y$ is {\it sequence-covering} if each convergent sequence with the
limit point of $Y$ is the image of some convergent sequence with the limit point of $X$. Following Lin \cite{Lin}, a mapping $f: X\to Y$ is a {\it mssc-mapping} (i.e., metrizably stratified strong compact mapping) if there exists a subspace $X$  of the product space $\prod_{n\in\NN} X_n$ of a family $\{ X_n\}_{n\in\NN}$ of metric spaces satisfying the following condition: for each $y\in Y$, there exists an open neighborhood sequence $\{ V_i\}$ of $y$ such that each $\mathrm{cl}\left(p_i(f^{-1}(V_i)\right)$ is compact in $X_i$, where $p_i :\prod_{n\in\NN} X_n \to X_i$ is the projection. Li \cite{Li} characterized $\aleph$-spaces as follows:
\begin{theorem}[\cite{Li}]
A regular space $X$ is an $\aleph$-space if and only if $X$ is a  sequence-covering mssc-image of a metric space.
\end{theorem}

\begin{question}
Characterize analogously strict $\aleph$-spaces and $\Pp$-spaces.
\end{question}

Denote by $C_p(X)$ the space $C(X):=C(X,\mathbb{R})$ endowed with the pointwise topology.
Sakai \cite{Sak}  proved that $C_p(X)$ has countable $cs^\ast$-character if and only if $X$ is countable. It is well known that $\chi(C_p(X))= |X|$ and hence $cs^\ast(C_p(X)) \leq |X|$ for every infinite $X$.
\begin{question}
Is $cs^\ast(C_p(X))=|X|$ for every infinite Tychonoff space $X$?
\end{question}
If this question has a positive answer  then also $cp(C_p(X))=ck(C_p(X))=|X|$ by Corollary \ref{c1}.
It is well known that $C_p(X)$ is $b$-Baire-like for every Tychonoff space $X$.
In \cite{GaK} we proved that a $b$-Baire-like locally convex space $E$ is metrizable if and only if $E$ has countable $cs^\ast$-character.
\begin{question}
Let $E$ be a $b$-Baire-like locally convex space. Is $cs^\ast(E)=\chi(E)$?
\end{question}

\section{Acknowledgements}

The authors thank  to Taras Banakh for pointing out Example \ref{exa-A-non-CS}.

\bibliographystyle{amsplain}

\end{document}